\documentclass{article}

\title{Computing twisted conjugacy classes in free groups using nilpotent quotients}
\author{P. Christopher Staecker\thanks{\emph{Email:} cstaecker@messiah.edu} \\
Assistant Professor, Dept. of Mathematical Sciences, \\
Messiah College, Grantham PA, 17027.}

\usepackage{multirow, hyperref, named, amsthm, amsmath, amsfonts}

\newtheorem{thm}{Theorem}
\newtheorem{prop}[thm]{Proposition}

\theoremstyle{definition}
\newtheorem{exl}[thm]{Example}
\newtheorem{rem}[thm]{Remark}

\numberwithin{thm}{section}

\renewcommand{\phi}{\varphi}
\newcommand{\pd}{\partial}

\newcommand{\RT}{\text{\textit{RT}}}
\DeclareMathOperator{\Fix}{Fix}
\newcommand{\magma}{\textsc{Magma}}

\newcommand{\Reid}{\mathcal{R}}
\newcommand{\Z}{\mathbb{Z}}

\newcommand{\lift}{\widetilde}
\newcommand{\abel}{\bar}
\newcommand{\nilp}{\widehat}

\newcommand{\journalversion}[1]{}

\begin{document}

\bibliographystyle{named}
\maketitle

\begin{abstract}
There currently exists no algebraic algorithm for computing twisted
conjugacy classes in free groups. We propose a new technique for
deciding twisted conjugacy relations using nilpotent quotients. Our
technique is a generalization of the common abelianization method, but
admits significantly greater rates of success. We present experimental results
demonstrating the efficacy of the technique, and detail how it can be
applied in the related settings of surface groups and doubly twisted
conjugacy. 
\end{abstract}

\section{Introduction}
Given two elements $x$ and $y$ in a group $G$ and an endomorphism
$\phi: G \to G$, we say that $x$ and $y$ are \emph{twisted conjugate}
if there is some $z$ such that 
\[ x = \phi(z) y z^{-1}. \]
Twisted conjugacy is a generalization of the ordinary conjugacy
relation in groups, and the computation of twisted conjugacy classes
is a problem of considerable difficulty for many groups $G$.

Computing twisted conjugacy classes (also called ``Reidmeister
classes'') is of interest in various algebraic contexts. Our own
approach will be motivated by Nielsen fixed point theory, though other
motivations exist (see \cite{bmmv06}, in which the twisted conjugacy
problem in free groups arises naturally in the context of the ordinary
conjugacy problem in certain other groups).

Extant algebraic techniques for computing twisted conjugacy classes
are \emph{ad hoc} in nature, and the goal of this paper is to present a new
technique which is more generally applicable (though still not in
general algorithmic), along with experimental results demonstrating
its success.

Our technique is an extension of the common abelianization
technique. Viewing the abelianization as the first nilpotent quotient,
we show that twisted conjugacy classes can be distinguished with
increasing success rates when projected into nilpotent quotients of
increasing nilpotency class. These projections, and the computations
necessary to compute twisted conjugacy in nilpotent groups, can be
fairly intensive, and as such we have implemented our technique in
the \magma\ programming language. \journalversion{Our implementation
  has been included as a supplement to this paper at \url{http://www.expmath.org/???}}

In Section \ref{nielsen} we give the motivation for the twisted
conjugacy problem from Nielsen theory, Section \ref{quotients} is an
outline of our technique, Section \ref{examples} gives some examples,
Section \ref{rates} presents our experimental results, and Sections
\ref{surface} and \ref{doublesection} show how the technique can be
adapted into two natural generalizations of the main problem.

The bulk of this work was completed as part of the author's doctoral
dissertation, advised by Robert F. Brown. The content in this
paper was also advised and suggested by Peter Wong. The author
wishes to acknowledge their help and support, as well as helpful
input from Seungwon Kim, Armando Martino, and Evelyn Hart.

\section{Nielsen fixed point theory}\label{nielsen}
Our principal motivation for studying the twisted conjugacy problem is
Nielsen fixed point theory (standard references are \cite{jian83} and
\cite{kian80}). Given a map $f:X \to X$ of a 
space with universal covering space $\lift X$ and projection
map $p:\lift X \to X$, the fixed points of $f$ are partitioned into
classes of the form $p(\Fix(\lift f))$, where $\lift f$ is a lift of
$f$. If a single lift $\lift f$ is fixed once and for all, each fixed
point class can be expressed as $p(\Fix \alpha^{-1} \lift f)$ for some
$\alpha \in \pi_1(X)$. A central problem in Nielsen fixed point theory 
is to determine the number of ``essential'' fixed point
classes of a mapping. This number gives a lower bound for the minimal
number of fixed points for maps in the homotopy class of $f$, and in
the case when $X$ is a manifold of dimension at least 3, this
``Nielsen number'' is in fact equal to the minimal number of fixed
points.

A fundamental question when counting fixed point classes is the
following: given two elements $\alpha, \beta \in \pi_1(X)$, when is
$p(\Fix(\alpha^{-1} \lift f)) = p(\Fix(\beta^{-1} \lift f))$? The question can
be answered with an elementary argument in covering space theory. Two
such fixed point classes are equal if and only if there is some
$\gamma \in \pi_1(X)$ with 
\[ \alpha = \phi(\gamma) \beta \gamma^{-1}, \]
where $\phi: \pi_1(X) \to \pi_1(X)$ is the map induced by $f$ on the
fundamental group. If the above holds, we say that the elements
$\alpha$ and $\beta$ are \emph{twisted conjugate} (if $\phi$ is the
identity map, this is the ordinary conjugacy relation). The 
twisted conjugacy classes of $f$ are also called \emph{Reidemeister
  classes}, and the set of such classes is denoted $\Reid(\phi)$.

In the case where $f$ is a selfmap of a compact surface with boundary,
the fundamental group $\pi_1(X)$ will be a free group, say $\pi_1(X) =
\langle g_1, \dots, g_k \rangle$. No general algorithm is known for computing
the Nielsen number of such a map (the case $k=1$ is classically known,
and an algorithm for the case $k=2$ has recently been developed by Yi and
Kim \cite{yk08}, extending work by Joyce Wagner \cite{wagn99}). Fadell
and Husseini, in \cite{fh83}, proved: 
\begin{thm}[Fadell, Husseini, 1983] \label{foxcalc}
Let $\phi$ be the map induced by $f$ on $\pi_1(X)$. Then the Nielsen
number of $f$ is the number of terms with nonzero coefficient in
\[ \RT(\phi) = \rho \left( 1 - \sum_{i = 1}^n \frac{\pd}{\pd g_i}
\phi(g_i) \right), \]
where $\pd$ is the Fox Calculus operator (see \cite{cf63}), and 
$\rho: \Z\pi_1(X) \to \Z\Reid(\phi)$ is the linear extension of the
projection into twisted conjugacy classes.
\end{thm}

The theorem above thus reduces the problem of computing the Nielsen
number to the application of the projection $\rho$, which requires
some algorithm for distinguishing twisted conjugacy classes. No such
algorithm exists in the literature, though Bogopolski,
Martino, Maslakova, and Ventura give an algorithm in \cite{bmmv06} for
the case where $\phi$ is an isomorphism. Their algorithm involves
the traintracks machinery of Bestvina and Handel \cite{bh92}. This paper
explores a purely algebraic technique which can be implemented by
existing computer algebra systems, and in some cases can be done by
hand. 

\section{Abelian and nilpotent quotients}\label{quotients}
Throughout this and the next two sections, let $G$ be a finitely generated
free group, and let $\phi:G \to G$ be an
endomorphism. For an element $g \in G$, let $[g]$ denote the twisted
conjugacy class of $g$. Our goal is to confirm or deny the equality
$[g] = [h]$ for two group elements $g$ and $h$.

Existing algebraic techniques for deciding equality of such classes
are surveyed in \cite{hart05}. One such technique is the algorithm of
Wagner \cite{wagn99}, which is applicable if the mapping $\phi$
satisfies the combinatorial ``remnant'' condition (this condition is
satisfied with probability approaching 1 as the word lengths of the
images of generators increase).
For general mappings (without remnant), distinguishing classes is
typically done by projecting into the abelianization $\abel G$ and
solving the twisted conjugacy relation there. 

Abelianization is often sucessful at showing that two classes are
distinct, but cannot be used to show that two classes are
equal. Projecting into nilpotent quotients is a natural generalization
of the abelianization technique, and we show how it can also furnish a
technique for equating classes.

We will use the commutator notation $[x,y] = xyx^{-1}y^{-1}$. 
Let $\gamma_n(G)$ be the terms of the lower central series, and let $\nilp
G_n = G / \gamma_n(G)$. Each of the groups $\nilp G_n$ are nilpotent
of class $n$. The abelianization is of course $\abel{G} =
\nilp{G}_1 = G/\gamma_1(G)$. In general, we will use a bar to indicate
projection of elements into the abelianization, and a hat to
indicate projection into $\nilp G_n$, with the value of $n$ to be
understood by context.

Computation in $\abel G$ is made easy by commutativity. For $n>1$, the 
groups $\nilp G_n$ are not commutative, but powerful commutation rules
make computation possible. The following easily verifiable
commutator identities hold in any group:
\begin{align*}
yx &= [y,x]xy, \\
[y,x] &= [x,y]^{-1}, \\
[xy,z] &= [x,[y,z]][y,z][x,z]
\end{align*}

These rules have a nicer form in a class 2 nilpotent group, where all
commutators will freely commute:
\begin{prop}\label{class2rules}
If $G$ is a class 2 nilpotent group, then, for any $x,y,z \in G$, we
have
\begin{align}
yx &= xy[x,y]^{-1}, \label{commutation} \\
[y,x] &= [x,y]^{-1}, \label{cominv} \\
[xy,z] &= [x,z][y,z]. \label{comsplit}
\end{align}
\end{prop}

In a class 2 nilpotent group, we may use (\ref{commutation}) to
exchange the order of any non-commutator elements. Using (\ref{comsplit})
we may write any commutator as a product of commutators of
generators. Having reduced all commutators to commutators of
generators, we may use (\ref{cominv}) to ensure that the generators
appear in a prescribed order. Viewed as a set of word rewriting rules,
Proposition \ref{class2rules} suggests that there will be some sort of
normal form in nilpotent groups which can be used to compare words. 

The desired normal form is provided by a theorem of P. Hall. Before
stating the theorem, we give some terminology and notation, following
\cite{hall88}. 

For a free group $G$, consider the elements which can be formed by
taking the closure of the generator set under the commutator
operation. Of these elements, the generators are referred 
to as \emph{weight 1 commutators}, and the weight of any non-generator
element is defined to be the sum of the weights of the elements of
which it is a commutator.

Hall showed that words in $\nilp G_n$ can be given in a normal form
consisting of a product of certain \emph{basic commutators} of weight
$n$ or less given in some proscribed order. The construction of
the basic commutators is somewhat involved, and we refer to
\cite{hall88} and \cite{mks76} for the details. For 
the purpose of the examples in this paper, it is sufficient to know
that in a group of rank 2, say $G = \langle a, b\rangle$: the basic weight 1
commutators are $a$ and $b$, and the only basic weight 2 commutator
is $[a,b]$ (the commutator $[b,a]$ is not basic, since it is
expressable as $[a,b]^{-1}$).
We also refer to Theorem 5.11 of \cite{mks76} which gives a
combinatorial formula due to Witt for $C_n$, the number of basic
weight $n$ commutators:
\[ C_n = \frac1n \sum_{d|n}\mu(d)k^{n/d}, \]
where $\mu$ is the M\"obius function, and $k$ is the number of
generators of $G$.

\begin{thm}[P. Hall, 1957]
For any $x \in G$, we can write the projection $\nilp x \in \nilp G_n$ as
\[ \nilp x = \prod_{i} \nilp c_i^{k_i}, \]
where $\{c_i\}$ is the sequence of basic commutators of weight less
than or equal to $n$.

This form for $\nilp x$ is unique up to the ordering of the weight $n$ basic
commutators, and we call this the \emph{Hall normal form} for $\nilp x$.
\end{thm}

Since $\phi(\gamma_n(G)) \subset \gamma_n(G)$, there is a well defined
quotient mapping $\nilp \phi:\nilp G_n \to \nilp G_n$. Thus it
makes sense to ask, for $h, g \in G$, whether or not $[\nilp h] =
[\nilp g]$ in $\nilp G_n$, that is, whether or not there is some $z
\in \nilp G_n$ with 
\[ \nilp h = \nilp \phi(z) \nilp g z^{-1}. \]
If $[\nilp h] \neq [\nilp g]$ in $\nilp G_n$, then we know that $[h]
\neq [g]$ in $G$. 

\section{Some examples}\label{examples}
We begin with a sample computation by hand, showing how the Hall
normal form can be used to solve twisted conjugacy relations.

\begin{exl}
We will compute the Nielsen number of the map on a surface
with fundamental group $G = \langle a, b\rangle$ which induces the
homomorphism: 
\[ \phi: \begin{array}{rcl} a & \mapsto & ab \\
  b & \mapsto & b^2 a^4 \end{array} \]

Theorem \ref{foxcalc} gives
\[ \RT(\phi) = \rho(-1 - b) \]
and thus we need only decide whether or not $[1] = [b]$. First we
attempt to equate these classes in the abelianization $\abel
G$. Writing elements additively, an element $z\in \abel G$ is of
the form $z = n \abel a + m \abel b$, and we wish to solve
\[ 0 = \abel \phi(z) + \abel b - z. \]
We compute $\abel \phi(z) = n(\abel a + \abel b) + m(4 \abel a + 2 \abel
b)$ and $-z = -n \abel a - m \abel b$, and the above equation becomes
\[ -\abel b = 4m \abel a + (n + m)\abel b, \]
and we can solve for $n$ and $m$ to find that $z = -\abel a$ is a
solution. 

\begin{rem}\label{candremark}
Given that $1$ and $b$ are twisted conjugate in the abelianization by
the element $-\abel a$, we might hope that $1$ and $b$ are twisted
conjugate in the group $G$ by the element $a^{-1}$. This is not the case,
however, as $\phi(a^{-1}) b a = b^{-1}a^{-1}ba$. The possibility
remains, however, that these elements are twisted conjugate by some
more complicated word which abelianizes to $-\abel a$.
\end{rem}

Having failed to decide the twisted conjugacy after a check in the
abelianization, we proceed to the class 2 nilpotent quotient $\nilp
G_2$. Any $z \in G_2$ is of the form $z = \nilp a^n \nilp b^m
[\nilp a, \nilp b]^k$. We wish to solve
\[ 1 = \nilp \phi(z) \nilp b z^{-1}. \]
We already know by the above calculation, however, that any such
element $z$ must abelianize to $\nilp a \in \nilp G$ in order to satisfy
the twisted conjugacy equation. Thus we may assume that $z = \nilp a^{-1}
[\nilp a, \nilp b]^k$. Now we compute
\[ z^{-1} = [\nilp a, \nilp b]^{-k} \nilp a = \nilp a [\nilp a,
  \nilp b]^{-k}, \]
and 
\begin{align*} \nilp \phi(z) &= \nilp b^{-1} \nilp a^{-1} [\nilp a \nilp b,
	\nilp b^2 \nilp a^4]^k = \nilp a^{-1} \nilp b^{-1} [\nilp a^{-1}, \nilp
  b^{-1}] [\nilp a, \nilp b]^{2k} [\nilp b, \nilp 
  a]^{4k} = \nilp a^{-1}\nilp b^{-1} [\nilp a,\nilp b]^{-2k+1}, \end{align*}
where we have used the rules of Proposition \ref{class2rules},
together with the identity $[x^i, z] = [x, z]^i$, which follows from
setting $x=y$ in identity (\ref{comsplit}). 

We are now ready to test the twisted conjugacy equation above. The
right hand side is:
\[ \nilp \phi(z) \nilp b z^{-1} = \nilp a^{-1} \nilp b^{-1} [\nilp a, \nilp
  b]^{-2k-8} \nilp b \nilp a [\nilp a, \nilp b]^{-k} = [\nilp a, \nilp
  b]^{-3k - 8}. \]
Setting this equal to 1 requires that $-3k - 8 = 0$, which is
  impossible for $k \in \Z$. Thus there can be no such $z \in \nilp
  G_2$, and so $[1] \neq [b]$ and the Nielsen number is 2.
\end{exl}

The example above involved a computation first in $\abel G$ and then in
$\nilp G_2$. In each step, there are essentially two types of
computational operations involved. The first is the term
rewriting to obtain the Hall normal form, which was fairly easy in
this case but in general can be quite tedious (though completely
algorithmic). The second is finding the solution to a linear system,
which was used to solve for $n$ and $m$ in the $\abel G$ step, and
used to solve for $k$ in the $\nilp G_2$ step (in a free group with
more generators, this would have been a linear system with more than
one equation).  

As is to be expected, computation of twisted conjugacy classes in many
cases may require checks in $\nilp G_n$ for $n>2$. Such examples can
easily be constructed by a computer search. 
\begin{exl}
Let $G = \langle a, b\rangle$, and let
\[ \phi: \begin{array}{rcl} a & \mapsto & aba^{-1} \\ b & \mapsto & a^{-2}b^4
\end{array} \]
Theorem \ref{foxcalc} gives
\[ \RT(\phi) = \rho(aba^{-1} - a^{-2} - a^{-2}b - a^{-2}b^2 -
a^{-2}b^3). \]
It can be verified that all five of the above terms are twisted
conjugate in $\nilp G_n$ for $n \in \{1, 2, 3\}$, but none are twisted
conjugate in $\nilp G_4$. Thus the Nielsen number is 5.
\end{exl}

We now turn to the question of verifying equality of twisted conjugacy
classes. Applying the process described above to two words which are
in fact twisted conjugate will result in a non-terminating sequence of
computations in the groups $\nilp G_n$, each time resulting in
solutions for the various elements above labeled $z$.

To avoid such an infinite computation, we propose a technique in the spirit of
Remark \ref{candremark}. In the abelianization $\abel G$, if a
solution element $z$ is obtained, a sequence of ``candidates'' for
testing twisted conjugacy in $G$ is constructed by producing all
possible reorderings of the generators appearing in $z$. Thus if we
obtain an element $z = \abel a - 2\abel b$, our list of candidates will
be
\[ ab^{-2}, b^{-1}ab^{-1}, b^{-2}a. \]
Each of these elements can be tested by twisted conjugation as a
candidate for realizing the twisted conjugacy in $G$.

A similar process can be carried out in $\nilp G_n$ for $n>1$: after
obtaining $z \in \nilp G_n$, our list of candidates is obtained by
inserting the weight $n$ basic commutators appearing in $z$ in all possible
orderings and in various forms into each of our candidates previously
obtained in our check from $\nilp G_{n-1}$.

\begin{exl}
Let $G = \langle a, b \rangle$, and 
\[ \phi: \begin{array}{rcl} a & \mapsto& a^2ba \\ b &\mapsto & b^2 a
\end{array} \]
We will use the above candidates checking construction to show that
$[a] = [a^2 b]$ (the elements $a^2 b$ and $a$ appear in $\RT(\phi)$).

Our check in $\abel G$ shows that the two elements are twisted
conjugate by $-\abel b$. Thus our only candidate from $\abel G$ is the
element $b^{-1}$, but a computation shows that 
\[ \phi(b^{-1}) (a^2 b) b = a^{-1}b^{-2}a^2 b^2 \neq a. \]
Now we do a check in $\nilp G_2$, and we find that the two elements are
twisted conjugate by $\nilp b^{-1} [\nilp a, \nilp b]^{-1} \in \nilp G_2$. Now
our list of candidates is:
\begin{align*}
g_1 &= b^{-1} [a,b]^{-1} = ab^{-1}a^{-1} \\
g_2 &= b^{-1} [a^{-1},b^{-1}]^{-1}= b^{-2}a^{-1}ba \\
g_3 &= b^{-1} [a^{-1},b] = b^{-1}a^{-1}bab^{-1} \\
g_4 &= b^{-1} [a, b^{-1}] = b^{-1}ab^{-1}a^{-1}b \\
g_5 &= [a,b]^{-1} b^{-1} = bab^{-1}a^{-1}b^{-1} \\
g_6 &= [a^{-1}, b^{-1}]^{-1} b^{-1} = b^{-1}a^{-1}bab^{-1} \\
g_7 &= [a^{-1},b] b^{-1} = a^{-1}bab^{-2} \\
g_8 &= [a, b^{-1}] b^{-1} = ab^{-1}a^{-1}
\end{align*}
Computing each of $\phi(g_i)a^2 b (g_i)^{-1}$ reveals that $\phi(g_1)
a^2 b g_1^{-1} = a$, and so $[a] = [a^2 b]$.
\end{exl}

Note that in the above example we inserted the basic commutator
$[a,b]^{-1}$ into the word $b^{-1}$ in each of two positions in one of
four forms, these being the four versions of this commutator in $G$
which become $[\nilp a, \nilp b]^{-1}$ when projected into $\nilp
G_2$. Constructing the various forms of a weight 3 commutator to use
in a list of candidates would be cumbersome, and we do not attempt
this construction for nilpotency class higher than 2.

As an alternative to the candidates construction procedure, a more
pedestrian approach is always available: after a check for twisted
conjugacy in $\nilp G_n$, use as list of candidates all words in $G$ of
word length $n$. This produces a different list of
candidates from that described above, but is guaranteed to find an
element realizing the twisted conjugacy if $n$ is sufficiently high.

\section{Success rates}\label{rates}
The technique exhibited in the examples above can be summarized as
follows: Starting 
with $n=1$, write an expression for a generic element $z \in \nilp
G_n$ and the element $\nilp \phi(z)$ in Hall normal form.
Solve a linear system to decide if the
elements are twisted conjugate in $\nilp G_n$. If the elements are not
twisted conjugate in 
$\nilp G_n$, then the elements are not twisted conjugate in $G$. If the
elements are twisted conjugate in $\nilp G_n$ by a unique element $z$, then
construct a finite list of candidates (either intelligently by using
the structure of $z$ or in a brute force manner by taking all words of
length $n$) for testing twisted conjugacy in $G$. If all of these
candidates fail, then increment $n$ and repeat the above.

This technique will decide any given twisted conjugacy
problem provided that the following statement is true: If $\phi:G \to
G$ is a map on a free group, and $g$ and $h$ are two elements of $G$
which are not twisted conjugate in $G$, then there is some $n$ for
which $g$ and $h$ are not twisted conjugate in $\nilp
G_n$. Thus any non-twisted-conjugate elements will be detected as such
in $\nilp G_n$ for some $n$. Such a statement
does not hold in general, though, as the following argument shows.

\begin{prop}\label{badnews}
For any $\phi:G \to G$, if $g,h\in G$ are words such that $\phi^n(g) \in
h\gamma_n(G)$ for all $n$, then $[\nilp g] = [\nilp h]$ in $\nilp G_n$
for all $n$.
\end{prop}
\begin{proof}
The proof is based on the fact that $[\phi(x)] = [x]$ for any element
$x$, since $\phi(x) = \phi(x) x x^{-1}$. Iteration of the map gives
$[\phi^n(x)] = [x]$ for any $n$.

Now for our element $g$, we have $\nilp \phi^n(\nilp g) = \nilp h$ in
$\nilp G_n$, and so in particular $[\nilp \phi^n(\nilp g)] = [\nilp
h]$. But $[\nilp \phi^n(\nilp g)] = [\nilp g]$ by the above, and so
we have $[\nilp g] = [\nilp h]$.
\end{proof}

We can use the above to build homomorphisms $\phi: G \to G$ with the
property that there are words $g, h \in G$ with $[g]\neq [h]$ but
$[\nilp g] = [\nilp h]$ in every nilpotent quotient $\nilp G_n$.

\begin{exl}
Let $G = \langle a, b\rangle$, and let $\phi:G \to G$ be the map
\[ \phi: \begin{array}{rcl} a & \mapsto & [b,a] \\
b & \mapsto & a^{-1}b \end{array} \]
Now we have $\phi(\gamma_n(G)) \subset \gamma_{n+1}(G)$ for $n>0$,
since $\phi$ replaces $a$ with a weight 2 commutator and all
commutators in $G$ involve the element $a$. Since $\phi(a) \in
\gamma_1(G)$, we have $\phi^n(a) \in \gamma_n(G)$ for all $n$, and
thus that $[\nilp a] = [1]$ in $\nilp G_n$ for all $n$.

But $\phi$ is a mapping with remnant, and Wagner's algorithm can be
used to show that in fact $[a] \neq [1]$ in $G$.
\end{exl}

Though the above example shows that nilpotent quotients cannot 
be used directly to solve twisted conjugacy problems in
all cases, the technique gives very good success rates in experimental
testing. We have created an implementation\footnote{Available from the
  author's web site:
  \url{http://www.messiah.edu/~cstaecker}\journalversion{, and included as a supplement to this paper at
	\url{http://www.expmath.org/???}}} 
of the process in the computational algebra system \magma\ \cite{magma}. 

There are two ways that an implementation of this techinque
can fail to decide any given twisted conjugacy problem. One type of
failure is when the linear system which arises in the twisted
conjugacy computation in $\nilp G_n$ has infinitely many solutions. In
such a case the check in $\nilp G_{n+1}$ will require finding an
integral solution of a polynomial system, which will in general be
difficult. 
Such a failure can only occur when the coefficients in the linear
system give a singular matrix, which we expect to occur relatively
infrequently.

A second type of failure is that the implementation exhausts its
resources in computing the required Hall normal forms in $\nilp G_n$. Even for
groups of 2 or 3 generators, human computation of the Hall form in
class 3 or 4 is barely feasible. Since the number of basic commutators
of weight $n$ grows exponentially in $n$ we expect that any computer
implementation could conceivably exhaust its resources before
detecting that two given twisted conjugacy classes are indeed distinct. We
expect this type of failure to occur increasingly in free groups with
large numbers of generators.

Tables \ref{failtypes} and \ref{abelwag} give experimental results of
application of the 
nilpotent quotients technique to 10,000 randomly generated mappings on the
free group on $k$ generators for $k=2,3,4$. In each case, a number $l
\in \{2,3,4,5\}$ is chosen and a mapping is generated by assigning the
image of each generator to be a randomly chosen word of length at most
$l$. Theorem \ref{foxcalc} is applied to give a list of group elements
which must be divided into their twisted conjugacy classes. A
sucessful computation is one which is able to decide the twisted
conjugacy of these elements.

\begin{table}
\begin{center}
\begin{tabular}{|c|c| r @{.} l | r @{.} l | r @{.} l | r @{.} l | r
	@{.} l|}
\hline
$k$ & $l$ & 
\multicolumn{2}{|c|}{Success} &
\multicolumn{2}{|c|}{Matrix failure} & 
\multicolumn{2}{|c|}{Complexity failure} &
\multicolumn{2}{|c|}{Avg. depth} &
\multicolumn{2}{|c|}{$\sigma$} \\
\hline \hline
\multirow{4}{*}{2}
  & 2 & 94&27\% &  4&30\% &  2&97\% & 1&10 & 0&28 \\
  & 3 & 90&46\% &  6&66\% &  4&09\% & 1&13 & 0&34 \\
  & 4 & 85&91\% &  7&75\% &  9&06\% & 1&18 & 0&42 \\
  & 5 & 84&20\% &  9&89\% &  8&66\% & 1&17 & 0&38 \\
\hline
\multirow{4}{*}{3}
  & 2 & 87&45\% & 11&26\% &  4&38\% & 1&18 & 0&45 \\
  & 3 & 85&60\% & 13&77\% &  4&74\% & 1&17 & 0&42 \\
  & 4 & 85&20\% & 12&38\% &  6&19\% & 1&21 & 0&47 \\
  & 5 & 84&46\% & 13&36\% &  6&29\% & 1&23 & 0&45 \\
\hline
\multirow{4}{*}{4}
  & 2 & 88&87\% & 10&95\% &  3&48\% & 1&13 & 0&41 \\
  & 3 & 85&60\% & 13&77\% &  4&74\% & 1&16 & 0&42 \\
  & 4 & 84&51\% & 14&56\% &  5&87\% & 1&21 & 0&45 \\
  & 5 & 84&13\% & 15&04\% &  5&64\% & 1&21 & 0&42 \\
\hline
\end{tabular}
\end{center}
\caption{Results of testing for success rates on
random mappings of word length $l$ on the free group on $k$
generators.\label{failtypes}}  
\end{table}

Table \ref{failtypes} gives the rates of each of the two types of failure
above. A ``matrix failure'' is declared when the linear system
computation results in infinitely many solutions. A ``complexity
failure'' is declared when the nilpotency class reaches 5, as this is
the level at which the computation of the Hall normal form becomes
difficult for our implementation (run on a personal computer, current
in 2005). Since a single mapping can trigger both types of error
(computation of $\RT(\phi)$ in general requires several twisted
conjugacy decisions), the row percentages may not sum to $100\%$.
The column labeled ``average depth'' gives the average
nilpotency class required to distinguish twisted conjugacy classes in
these random mappings. A depth of $n$ indicates that a check in $\nilp
G_n$ was necessary. The column labeled $\sigma$ gives the standard
deviations of the depths.

\begin{table}
\begin{center}
\begin{tabular}{|c|c| r @{.} l | r @{.} l | r @{.} l |}
\hline
$n$ & $l$ & 
\multicolumn{2}{|c|}{Nilpotent quotients} &
\multicolumn{2}{|c|}{Abelianization} & 
\multicolumn{2}{|c|}{Wagner's Alg.} \\
\hline \hline
\multirow{4}{*}{2}
  & 2 & 94&27\% & 82&75\% & 41&80\% \\
  & 3 & 90&46\% & 70&45\% & 48&32\% \\
  & 4 & 85&91\% & 59&07\% & 54&71\% \\
  & 5 & 84&20\% & 51&14\% & 59&83\% \\
\hline
\multirow{4}{*}{3}
  & 2 & 90&98\% & 77&96\% & 15&54\% \\
  & 3 & 87&45\% & 65&99\% & 24&70\% \\
  & 4 & 85&20\% & 55&17\% & 33&72\% \\
  & 5 & 84&46\% & 47&43\% & 41&45\% \\
\hline
\multirow{4}{*}{4}
  & 2 & 88&87\% & 76&50\% &  5&84\% \\
  & 3 & 85&60\% & 64&76\% & 13&67\% \\
  & 4 & 84&51\% & 54&63\% & 22&63\% \\
  & 5 & 84&13\% & 47&81\% & 30&18\%\\
\hline
\end{tabular}
\end{center}
\caption{Comparison of success rates of various techniques on
random mappings of word length $l$ on the free group on $k$
generators.}\label{abelwag}  
\end{table}

Table \ref{abelwag} gives our sucess rates compared to the existing
techniques of abelianization and Wagner's algorithm. The technique
used for data in the ``Abelianization'' column uses only the
abelianization for distinguishing classes, and uses only the identity
$[\phi(g)] = [g]$ for equating classes (this is the general strategy
employed in \cite{hart05}). The column labeled ``Wagner's alg.''
records the percentages of maps satisfying Wagner's remnant condition,
for which her algorithm will apply (that the percentages grow in $l$
is expected in light of Theorem 3.7 of \cite{wagn99}).

\section{Surfaces without boundary}\label{surface}
The nilpotent quotients process can be used with minimal modifications
when $G$ is the fundamental group of a compact hyperbolic surface
without boundary. Work of Fadell and Husseini in \cite{fh83} together
with a technique by Davey, Hart, and Trapp \cite{dht96} reduce the
computation of the Nielsen number on compact hyperbolic surfaces
without boundary to the computation of twisted conjugacy classes in
the fundamental group.

Wagner's technique does not apply if $G$ is not a free group, and
neither will the techniques of \cite{bmmv06}. Because
surface groups are easily expressed in terms of commutator relations,
we can use the nilpotent quotients technique with only trivial
modifications in this setting.

The only modification that must be made is to the precise structure of
the Hall normal form. For instance, if $G$ is the fundamental group of
the genus 2 compact surface, then $G$ has group presentation $G =
\langle a, b, c, d | [a,b][c,d]=1 \rangle$. The Hall normal form of an
element of (e.g.) $\nilp G_2$ can be obtained by applying commutation
rules as if $G$ were the free group on 4 generators, along with an
additional rule that $[c,d]^{-1} = [a,b]$. It is convenient for us
that the group structure of $G$ is so compatible with the Hall normal
form. 

The case of surfaces without boundary is somewhat more difficult to
implement in \magma, as it involves computations in finitely-presented
rather than free groups. The capabilities of \magma\ are somewhat
lacking in this regard-- in particular \magma\ (as of version 2.13-15) 
is unable to reliably 
solve the word problem in a surface group (although this word problem
is solvable). This causes the candidates checking process to return
false negatives, as the implementation may not recognize when two
elements are actually equal. Statistics such as those in Table
\ref{failtypes} are also difficult to produce in this setting as it is
difficult to generate random endomorphisms of surface groups.

\section{Doubly twisted conjugacy}\label{doublesection}
We conclude with a brief discussion of how our technique can be
applied to the \emph{doubly twisted conjugacy} relation: Given two
maps $\phi, \psi: G \to H$ and two elements $h, k \in H$, we say that
$h$ and $k$ are (doubly) twisted conjugate (we write $[h]=[k]$) if
there is an element $g \in G$ with 
\[ h=\phi(g)k\psi(g)^{-1}. \]
This relation is fundamental in Nielsen coincidence theory (see
\cite{gonc05}), playing the same role as ordinary twisted conjugacy in
fixed point theory.

For any $n$, the maps $\phi$ and $\psi$ will induce maps $\nilp
\phi, \nilp\psi: \nilp G_n \to \nilp H_n$, and the doubly twisted
conjugacy relation can in principle be solved by using Hall normal
forms in $\nilp G_n$ and $\nilp H_n$ just as in the ordinary twisted
conjugacy problem.

\begin{exl}
Let $G = H = \langle a, b \rangle$, and let our maps be 
\[ \phi: \begin{array}{rcl} a &\mapsto &b^2 a \\
b &\mapsto &a^{-2} \end{array} 
\quad
\psi: \begin{array}{rcl} a &\mapsto &a^3 \\
b &\mapsto &a^{-1} \end{array}
\]
We will decide the twisted conjugacy of the elements $b$ and $b^{-1}$.

We begin with check in the abelianization, where any element $z
\in \abel G$ has the form $z = n\abel a + m\abel b$. We compute that
$\abel \phi(z) = (n-2m)\abel a + 2n \abel b$ and $- \abel \psi(z) =
(-3n+m) \abel a$, and thus we have
\[ \abel\phi(z) - \abel b - \abel \psi(z) = (-2n -m)\abel a +
(2n-1)\abel b. \]
Equating this with $\abel b$ and solving gives $n=1$ and $m=-2$. This
solution in the abelianization gives three candidates for twisted
conjugacy:
\[ ab^{-2}, b^{-1}ab^{-1}, b^{-2}a, \]
but checking each shows that none of these realize the twisted
conjugacy in $H$. 

We proceed to the class 2 nilpotent quotient, where any element
$z \in \nilp G_2$ has the form $z = \nilp a^n \nilp b^m [\nilp a,
  \nilp b]^k$. Our computation in $\abel G$ shows that $n=1$ and $m=-2$,
simplifying our element to $z = \nilp a \nilp b^{-2} [\nilp a, \nilp
  b]^k$. We compute
\begin{align*}
\nilp \phi(z) &= \nilp b^2 \nilp a (\nilp a^{-2})^{-2} [\nilp b^2 \nilp
  a, \nilp a^{-2}]^k = \nilp a^5 \nilp b^2 [\nilp a^5, \nilp b^2]
  [\nilp b^2, \nilp a^{-2}]^k = \nilp a^5 \nilp b^2 [\nilp a, \nilp
  b]^{10+4k}, \\
\nilp \psi(z) &= \nilp a^3 (\nilp a^{-1})^{-2} [\nilp a^3, \nilp
  a^{-1}]^k = \nilp a^5,
\end{align*}
and so
\begin{align*}
\nilp \phi(z) \nilp b^{-1} \nilp \psi(z)^{-1} &= \nilp a^5 \nilp b^2
	  [\nilp a, \nilp b]^{10 + 4k} \nilp b^{-1} \nilp a^{-5} = \nilp
	  a^5 \nilp b \nilp a^{-5} [\nilp a, \nilp b]^{10+4k}\\
 &= \nilp b
	  [\nilp a^{-5}, \nilp b] [\nilp a, \nilp b]^{10+4k} = \nilp b
	  [\nilp a, \nilp b]^{5+4k}.
\end{align*}
Equating this with $\nilp b$ gives $5+4k =0$ which is impossible for
integral $k$. Thus $[b] \neq [b^{-1}]$.
\end{exl}

There is in the literature no analogue of Theorem \ref{foxcalc} in
coincidence theory, but presumably one may be available in the future,
and our technique is currently the only available
technique for distinguishing doubly twisted conjugacy
classes (no version of Wagner's algorithm is known in coincidence
theory, and the methods of \cite{bmmv06} do not extend in an obvious
way to doubly twisted conjugacy).

Table \ref{double} gives success rates for the technique applied to
10,000 randomly generated twisted conjugacy
relations. In each case, two random mappings of ``word length'' $3$
(the quantity labeled $l$ in Tables \ref{failtypes} and \ref{abelwag})
are generated from the free group on $k_1$ generators to the free
group on $k_2$ generators. Two random elements of the codomain group
are generated with word length at most 3, and the implementation attempts
to determine their twisted conjugacy. Entries in the table with no
digits to the right of the decimal point are exact figures, e.g. in
the case where $k_1 = 4$ and $k_2=3$ the depth was exactly 1 in each
of the 10,000 test cases, and there were exactly 0 complexity failures.

\begin{table}
\begin{center}
\begin{tabular}{|c|c| r @{.} l | r @{.} l | r @{.} l | r @{.} l | r
@{.} l | r @{.} l | r @{.} l |}
\hline
$k_1$ & $k_2$ & 
\multicolumn{2}{|c|}{Success} &
\multicolumn{2}{|c|}{Matrix failure} & 
\multicolumn{2}{|c|}{Complexity failure} &
\multicolumn{2}{|c|}{Average depth} &
\multicolumn{2}{|c|}{$\sigma$} \\
\hline \hline
\multirow{4}{*}{2}
  & 2 &  92&38\%  &   4&07\%  &   3&55\%  &  1&49 & 0&95 \\
  & 3 &  98&97\%  &   1&03\%  &   0& \%   &  1&09 & 0&28 \\
  & 4 &  99&78\%  &   0&22\%  &   0& \%   &  1&03 & 0&17 \\
  & 5 &  99&90\%  &   0&10\%  &   0& \%   &  1&01 & 0&12 \\
\hline
\multirow{4}{*}{3}
  & 2 &  30&53\%  &  69&47\%  &   0& \%   &  1&   & 0& \\
  & 3 &  92&32\%  &   5&83\%  &   1&85\%  &  1&41 & 0&80 \\
  & 4 &  98&80\%  &   1&20\%  &   0&  \%  &  1&08 & 0&27 \\ 
  & 5 &  99&71\%  &   0&29\%  &   0&  \%  &  1&03 & 0&18 \\
\hline
\multirow{4}{*}{4}
  & 2 &  14&56\%  &  85&44\%  &   0&  \%  &  1&   & 0& \\
  & 3 &  32&88\%  &  67&12\%  &   0&  \%  &  1&   & 0& \\
  & 4 &  91&98\%  &   7&13\%  &   0&89\%  &  1&33 & 0&66 \\
  & 5 &  98&50\%  &   1&15\%  &   0&  \%  &  1&08 & 0&27 \\
\hline
\end{tabular}
\end{center}
\caption{Success rates for doubly twisted conjugacy relations.
Random mappings
of word length 3 from the free group on $k_1$ generators to the free
group on $k_2$ generators were tested in deciding twisted conjugacy
between two random words of length at most 3.\label{double}}
\end{table}

Note that the technique is much less successful if the rank of the
domain is greater than the rank of 
the codomain. This is to be expected, as the technique will fail when
our linear system computation (always having more variables than 
equations if $k_1 > k_2$) yields infinitely many solutions. Note that
such cases are not handled by our \magma\ implementation, but could in
principle be done by hand. These would require finding integer
solutions to polynomial systems, and so we do not always expect
the computation to be successful, but particular examples may be
computable. 

Especially striking are the extremely high sucess rates when $k_2 >
k_1$. The vast majority of these twisted conjugacy
relations are decided in the abelianization (and with negative
result), as the
overdetermined linear systems are unlikely to have any
solutions. For example in the case of $k_1 = 3$ and $k_2 = 5$, an
equivalence in the abelianization would require our random elements to
satisfy a linear system of 5 equations and 3 unknowns. If the elements
are indeed equivalent in the abelianization, a further equivalence in
the class two nilpotent quotient would require our random data to
satisfy a linear system of 10 equations and 3 unknowns (the free group
on 5 generators has 10 basic weight 2 commutators, and the free
group on 3 generators has 3 basic weight 2 commutators).
This is not, of course, to say that twisted conjugate elements do not
occur in such cases where $k_2 > k_1$, but only that this occurs very
infrequently in the case of two words of length 3.

We have omitted from our testing the cases where $k_1$ or $k_2$ is
1. The twisted conjugacy relation is generally solvable in these
cases by hand.

Let $\phi, \psi: G \to H$ are maps of free groups where $G = \langle a
\rangle$, and let $h,k\in H$ be two words. Then the
twisted conjugacy problem is equivalent to finding some integer $n$
with 
\[ 1 = h^{-1}\phi(a)^n k \psi(a)^{-n}, \]
and this can typically be confirmed or denied by inspection.

If $\phi, \psi: G \to H$ are maps of free groups and $H$ has rank 1,
let $G = \langle a_1, \dots, a_n \rangle$, and note that
$\phi(\gamma_1(G)) \subset \gamma_1(H) = 1$ since $H$ is abelian. Thus
we have $\phi(uv) = \phi(vu[v,u]) = \phi(vu)$ for any words $v, u \in G$, and
similarly $\psi(uv) = \psi(vu)$. Thus, though there is no convenient
normal form for elements $z \in G$, we can say in general that 
\[ \phi(z) = \phi(a_1^{m_1} \dots a_n^{m_n}), \quad \psi(z) =
\phi(m_1^{k_1} \dots m_n^{k_n}) \]
by rearranging the generators of $z$.

Now to decide the twisted conjugacy of elements $h,k \in H$, we
examine the equation
\[ h = \phi(a_1^{m_1} \dots a_n^{m_n}) k \psi(a_1^{-m_1} \dots
a_n^{-m_n}), \]
and we will be able to determine whether or not the above has
solutions because it is an equation in the abelian group $H$. If the
above has a solution, then the elements are twisted conjugate, and if
not, they are not.

\end{document}